\documentclass[11pt]{amsart}
\usepackage[utf8]{inputenc}
\usepackage{mathrsfs}
\usepackage{amsfonts}
\usepackage{amssymb}
\usepackage{amsxtra}
\usepackage{dsfont}
\usepackage{color}
\usepackage{graphicx}
\usepackage[english,polish]{babel}
\usepackage{enumerate}
\usepackage[compress, sort]{cite}

\usepackage{newtxmath}
\usepackage{newtxtext}
\usepackage{bbm}

\usepackage[margin=2.5cm, centering]{geometry}
\usepackage[colorlinks,citecolor=blue,urlcolor=blue,bookmarks=true]{hyperref}
\hypersetup{
pdfpagemode=UseNone,
pdfstartview=FitH,
pdfdisplaydoctitle=true,
pdfborder={0 0 0}, 
pdftitle={Concrete representation of atomic $(F_4)$ filtrations},
pdfauthor={Maciej Rzeszut and Bartosz Trojan},
pdflang=en-US
}

\newcommand{\pl}[1]{\foreignlanguage{polish}{#1}}

\newcommand{\PP}{\mathbb{P}}
\newcommand{\RR}{\mathbb{R}}
\newcommand{\EE}{\mathbb{E}}

\newcommand{\NN}{\mathbb{N}}

\newcommand{\calF}{\mathcal{F}}
\newcommand{\calS}{\mathcal{S}}
\newcommand{\calG}{\mathcal{G}}
\newcommand{\calT}{\mathcal{T}}
\newcommand{\calI}{\mathcal{I}}
\newcommand{\calA}{\mathcal{A}}
\newcommand{\calB}{\mathcal{B}}

\newcommand{\ind}[1]{{\mathds{1}_{{#1}}}}

\newcommand{\at}{\mathrm{at}\,}

\newcommand{\supp}{\mathrm{supp}\,}

\newcommand{\beq}{\begin{equation}}
\newcommand{\eeq}{\end{equation}}
\newcommand{\vphi}{\varphi}

\allowdisplaybreaks

\theoremstyle{plain}
\newcounter{thm}

\newtheorem{main_theorem}[thm]{Theorem}

\newtheorem{theorem}{Theorem}[section]
\newtheorem{lemma}[theorem]{Lemma}
\newtheorem{conj}[theorem]{Conjecture}

\newtheorem{prop}[theorem]{Proposition}
\newtheorem{claim}[theorem]{Claim}

\theoremstyle{definition}
\newtheorem{remark}[theorem]{Remark}
\numberwithin{equation}{section}

\newcommand{\sprod}[2]{\langle {#1}, {#2} \rangle}

\author{Maciej Rzeszut}
\address{
	\pl{
	Maciej Rzeszut\\
	Instytut Matematyczny Polskiej Akademii Nauk\\
	ul. \'Sniadeckich 8\\
	00-696 Warszawa\\
	Poland
	}
	\&
	Department of Mathematics\\
	Weizmann Institute of Science\\
	P.O. Box 26\\
	Herzl St. 234\\
	7610000 Rehovot\\
	Israel
}
\email{maciej.rzeszut@gmail.com}

\author{Bartosz Trojan}
\address{
	\pl{
	Bartosz Trojan\\
	Instytut Matematyczny Polskiej Akademii Nauk\\
	ul. \'Sniadeckich 8\\
	00-696 Warszawa\\
	Poland}
}
\email{btrojan@impan.pl}
\thanks{The research was partialy supported by the National Science Centre, Poland, Grant 2016/23/B/ST1/01665}

\title{Concrete representation of atomic $(F_4)$ filtrations}

\begin{document}
\selectlanguage{english}

\begin{abstract} 
	We prove that for any martingale with respect to a biparameter atomic filtration satisfying $(F_4)$ 
	condition there is a martingale having the same joint distribution but with respect to the canonical $(F_4)$ 
	filtration. Even in one parameter case our result is an improvement of the theorem due to Montgomery-Smith, since
	the construction gives a morphism of filtrations and does not depend on underlying sequence.
\end{abstract}

\maketitle

\section{Introduction}
Let $(\Omega, \calF, \PP)$ be a probability space, i.e. $\Omega$ is a sample space with a $\sigma$-field $\calF$
and a probability measure $\PP$. A sequence of $\sigma$-fields $(\calF_i : i \in \NN_0)$ is called \emph{filtration} if
\[
	\calF_i \subset \calF_{i+1}, \qquad \text{for all}\quad i \in \NN_0.
\]
A model example of a filtration can be obtained by considering a product space
\[
	(S, \calS, \mu) = \bigotimes_{i = 0}^\infty (S_i, \calS_i, \mu_i)
\]
where each $(S_i, \calS_i, \mu_i)$ is a probability space. Then $\calF_i$ is set to be the $\sigma$-field generated by
the projection onto the first $i$ coordinates. The resulting sequence $(\calF_i : i \in \NN_0)$ will be called
the \emph{canonical} filtration on $(S, \calS, \mu)$.

Suppose that $(\calF_i : i \in \NN_0)$ is a filtration in a probability space $(\Omega, \calF, \PP)$.
A theorem due to Montgomery-Smith (see \cite[Theorem 3.1]{MS}) asserts that: For any sequence of random variables
$(f_n : 0 \leq n \leq N)$ on $(\Omega,\mathcal{F},\mathbb{P})$ which is a martingale with respect to
$(\mathcal{F}_n : 0 \leq n \leq N)$,
i.e. $\EE\left(f_m\mid\mathcal{F}_n\right)= f_n$ for $m\geq n$, there is a martingale sequence
$(\tilde{f}_n : 0 \leq n \leq N)$ of functions on $[0, 1]^N$ with respect to the canonical filtration and having the same
joint distribution as $(f_n : 0 \leq n \leq N)$. The construction is clever, but tailored to a given sequence
$(f_n : 0 \leq n \leq N)$. One of the goals of the present article is to remove this disadvantage, provided that $\Omega$
is discrete.

To achieve this we use the following observation: Suppose that there are two probability spaces $(S, \calS, \mu)$ and
$(T, \calT, \nu)$ equipped with families of $\sigma$-fields $(\calF_i : i \in \calI )$ and $(\calG_i : i \in \calI)$,
respectively. Assume that there is a mapping
\[
	\pi: (S, \calS, \mu) \rightarrow (T, \calT, \nu)
\]
so that
\begin{subequations}
\begin{equation}
	\label{eq:4}
	\pi^{-1}(U) \in \calF_i, \qquad\text{for all } U \in \calG_i \text{ and } i \in \calI,
\end{equation}
and
\begin{equation}
	\label{eq:5}
	\mu\big(\pi^{-1}(U) \big) = \nu(U), \qquad\text{for all } U \in \calT.
\end{equation}
\end{subequations}
Then $\pi$ induces a mapping
\footnote{
By $L^0(\Omega, \calF, \PP,X)$ we denote the space of equivalence classes of $\calF$-measurable functions with values in a measurable space $(X,\mathcal{X})$.}
\[
	\begin{aligned}
	\pi^*: L^0(T, \calT, \nu,X) &\longrightarrow L^0(S,\mathcal{S},\mu,X) \\
	f &\longmapsto f \circ \pi
	\end{aligned}
\]
that maps $\calF_n$-measurable functions to $\calG_n$-measurable functions preserving distributions, that is
\[
	\mu\big(\pi^*(f) \in B\big) = \nu\big(f \in B\big)
\]
for all $B\in \mathcal{X}$ and $f \in L^0(T,  \calG_n, \nu,X)$. The main theorem for one parameter case is the
following.
\begin{main_theorem}
	\label{main_thm:1}
	Let $(\calF_n : 1 \leq n \leq N)$ be a filtration in a discrete probability space $(\Omega, \calF, \PP)$.
	Then there is a sequence of probability spaces $\big((S_i, \calS_i, \mu_i) : i \in \NN_0 \big)$ 
	such that for any martingale sequence $(f_n : 1 \leq n \leq N)$ on $(\Omega, \calF, \PP)$ with respect
	to $(\calF_n : 1 \leq n \leq N)$ there is a martingale sequence $(\tilde{f}_n : 1 \leq n \leq N)$
	with respect to the canonical filtration of 
	\[
		(S, \calS, \mu) = \bigotimes_{i = 0}^\infty (S_i, \calS_i, \mu_i)
	\]
	having the same joint distribution as $(f_n : 1 \leq n \leq N)$.
\end{main_theorem}

A crucial step in the proof of Theorem \ref{main_thm:1} relies on our ability to equip the space of measurable functions
$\vphi:(S, \calS, \mu) \rightarrow (T, \calT, \nu)$ with a structure of a probability space such that the evaluation map
\[
	S\times T^S \ni (s,\vphi)\mapsto \vphi(s)\in T
\]
is measurable with respect to the product $\sigma$-field of $S\times T^S$ and enjoys certain quantitative properties.
This can be done in the case when $S$ and $T$ are discrete. Without this assumption, even the measurability of the
evaluation map is an obstacle (see \cite{Aum} for a detailed exposition of this problem). However, in many practical
applications (in particular, martingale inequalities) it is enough to approximate an arbitrary martingale with a
canonical one, thus reducing a given problem to a setting in which Theorem \ref{main_thm:1} applies. 
\begin{prop}
	\label{1parappr}
	Let $X$ be a separable Banach space. For any $X$-valued martingale $\left(f_n:0\leq n\leq N\right)$ on
	$\left(\Omega,\mathcal{F},\PP\right)$ with respect to a filtration $\left(\calF_n:0\leq n\leq N\right)$, and any
	$\varepsilon > 0$ there exists a filtration 
	$\left(\calG_n:0\leq n\leq N\right)$ such that $\calG_n$ is atomic, $\calG_n\subset \calF_n$, and 
	\[
		\left\|f_n-\EE\left(f_N\mid \calG_n\right)\right\|<\varepsilon
		\quad\text{for} \quad 0\leq n\leq N.
	\]
\end{prop}
\begin{proof}
	Due to separability of $X$, we may choose a countable subset $X_0\subset X$ and a function $u:X\to X_0$ such that
	$\|x-u(x)\|<\frac{\varepsilon}{2}$ for any $x\in X$. Take
	$\calG_n=\sigma\left(u\left(f_0\right),\ldots,u\left(f_n\right)\right)$. Each of the functions
	$u\left(f_0\right),\ldots,u\left(f_n\right)$ attains countably many values and is $\calF_n$-measurable, so $\calG_n$
	is atomic and $\calG_n\subset \calF_n$. Therefore
	\begin{align*}
		\left\|f_n-\EE\left(f_N\mid \calG_n\right)\right\|
		&= \left\|f_n-\EE\left(f_n\mid \calG_n\right)\right\| \\
		&\leq \frac{\varepsilon}{2}+\frac{\varepsilon}{2}
		+ \left\|u\left(f_n\right)-\EE\left(u\left(f_n\right)\mid\calG_n\right)\right\|=\varepsilon.\qedhere
	\end{align*}
\end{proof}
The advantage of the construction we use in the proof of Theorem \ref{main_thm:1} is the ability to extend it to
biparameter case which is the main result of the present paper. Let us recall that a double-indexed sequence of
$\sigma$-fields $(\calF_{i, j} : i, j \in N_0)$ is a biparameter filtration if for all $i, j \in \NN_0$,
\[
	\calF_{i, j} \subset \calF_{i, j+1}, \qquad \calF_{i, j} \subset \calF_{i+1, j}.
\]
The following condition was introduced in \cite{CW}, 
\footnote{$i \wedge j = \min\{i, j\}$}
\begin{equation}
	\tag{$F_4$}
	\label{eq:3}
	\EE\big(\EE(f \mid \calF_{i, j}) \mid \calF_{i', j'}\big) =
	\EE\big(f \mid \calF_{i \wedge i', j \wedge j'}\big),
\end{equation}
(or, equivalently, $\calF_{i, j+1}$ and $\calF_{i+1, j}$ are conditionally independent given $\calF_{i, j}$) as a natural 
minimum requirement for developing a theory of biparameter martingales. It allows to maintain a relatively rich structure,
see e.g. the seminal papers \cite{CW, WoZa}, monographs \cite{weisz1, weisz2, Imk} and references therein. Questions about 
biparameter analogues of theorems concerning one-parameter martingales range from copying proofs verbatim to
open problems. \par 
	
Let us motivate the study of biparameter martingales. Multiparameter harmonic analysis has been a subject of research
for more than 50 years, see e.g. \cite{Feff,FeffPiph,FeffSt,Gundy,GunSt,NaRiSt,NaSt,NaWa,RiSt,StSt,St}. It has started
with $\RR^d$ equipped with product dilation system \cite{GunSt,Gundy}. Later it was extended to more complex cases
exploiting curvature and culminating in the most general result recently obtain in \cite{StSt, St}. In one parameter
case, when the dilation parameter is restricted to dyadic numbers, the underlying structure is approximately a
martingale setup, see \cite{Ch, HyKa}. Hence, there is a very close analogy between harmonic analysis and martingale
theory. For example, parallel to real $H^1$ spaces there are martingale $H^1$ spaces. There is another connection,
namely, regular filtrations naturally appear while doing harmonic analysis over $p$-adic fields. Even closer relation
is visible by introducing dyadic shift \cite{Peter}. Many results that are correct in the real case
are also true in the martingale case, very often with cleaner proofs. 

The situation significantly changes in the multiparameter setup. The product case corresponds to very special
filtrations satisfying $(F_4)$ condition. Some work has been done to develop theory parallel to the real harmonic
analysis but very little was obtained without $(F_4)$ condition, see \cite{b1, CW, Gundy, Imk, weisz1, weisz2, WoZa}.
Pursuing this direction, in \cite{StTr}, we studied the filtration that naturally appears on the boundary of
$\tilde{A}_2$ affine buildings which is often isomorphic to $p$-adic Heisenberg group, thus does not satisfy
$(F_4)$ condition. This shows that one can hope to get the martingale counterpart of \cite{RiSt} or
even \cite{StSt}. 

However, $(F_4)$ condition is significantly weaker than being a product filtration.
In this paper our aim is to better understand the structure of filtrations satisfying $(F_4)$ condition. 
To achieve our goal we constructed a morphism of biparameter filtrations that allows us to map any martingale
on the original space to the martingale on the model space with respect to the canonical biparameter filtration.

The simplest example of biparameter filtrations satisfying
\eqref{eq:3} is a tensor of product filtrations. Namely, let $\big((S_i, \calS_i, \mu_i) : i \in \NN_0\big)$ and
$\big((T_j, \calT_j, \nu_j) : j \in \NN_0\big)$, be two sequences of probability spaces. In the product space
\[
	(S, \calS, \mu) \otimes (T, \calT, \nu)
\]
where
\[
	(S, \calS, \mu) = \bigotimes_{i=0}^\infty (S_i,\calS_i,\mu_i),
	\qquad\text{and}\qquad
	(T, \calT, \nu) = \bigotimes_{j=0}^\infty (T_j,\calT_j,\nu_j),
\]
we define $\calF_{i, j}$ to be the $\sigma$-field generated by the projections
\[
	(s, t) \mapsto \Big(\big(s_{i'} : 0 \leq i' \leq i\big), \big(t_{j'} : 0 \leq j' \leq j\big)\Big)
\]
which is the product of canonical filtrations on $(S, \calS, \mu)$ and $(T, \calT, \nu)$.
However, this is \emph{not} a universal model, since it is characterized by the property that 
$\mathcal{F}_{i,j}=\mathcal{F}_{i,j}^-$ where
\begin{equation}
	\label{eq:1}
	\begin{aligned}
	\mathcal{F}_{i,j}^-
	&=\sigma\Big(A \cup B : A \in \calF_{i-1, j} \text{ and } B \in \calF_{i, j-1} \Big) \\
	&=\calF_{i-1, j} \vee \calF_{i, j-1}.
	\end{aligned}
\end{equation}
A more general one is constructed from a double-indexed sequence of probability spaces 
$\big((S_{i, j}, \calS_{i, j}, \mu_{i, j}) : i, j \in \NN_0 \big)$. Namely, in the tensor product
\begin{equation}
	\label{eq:2}
	\bigotimes_{i = 0}^\infty \bigotimes_{j= 0}^\infty \left(S_{i,j},\mathcal{S}_{i,j},\mu_{i,j}\right),
\end{equation}
we set $\calF_{i, j}$ to be the $\sigma$-field generated by the projection
\[
	s \mapsto \big(s_{i', j'} : 0 \leq i' \leq i, 0 \leq j' \leq j\big).
\]
Heuristically, the whole space is generated by the surplus of $\mathcal{F}_{i,j}$ over $\mathcal{F}_{i,j}^-$,
and this will be the main idea of the proof of the fact that for any probability space equipped with an \eqref{eq:3}
filtration one can find a map $\pi$ from a product as \eqref{eq:2} having the desired properties
\eqref{eq:4} and \eqref{eq:5}. 
\begin{main_theorem}
	\label{main_thm:2}
	Let $(\calF_{i, j} : 1 \leq i \leq N, 1 \leq j \leq M)$ be a biparameter $\eqref{eq:3}$-filtration in a discrete
	probability space $(\Omega, \calF, \PP)$. Then there is a double-indexed sequence of probability spaces
	$(S_{i,j}, \calS_{i,j}, \mu_{i,j})$ such that for any double-indexed martingale sequence 
	$(f_{i, j} : 1 \leq i \leq N, 1 \leq j \leq M)$ on $(\Omega, \calF, \PP)$ with respect to the filtration
	$(\calF_{i, j} : 1 \leq i \leq N, 1 \leq j \leq M)$ there is a martingale sequence
	$(\tilde{f}_{i, j} : 1 \leq i \leq N, 1 \leq j \leq M)$ with respect to the canonical filtration of
	\[
		(S, \calS, \mu) = \bigotimes_{i = 1}^N \bigotimes_{j=1}^M \big(S_{i,j},\mathcal{S}_{i,j},\mu_{i,j}\big),
	\]
	having the same distribution as $(f_{i, j} : 1 \leq i \leq N, 1 \leq j \leq M)$.
\end{main_theorem}
Our purpose for developing Theorem \ref{main_thm:2} was to gain a deeper understanding of biparameter decoupling analogous
to that presented in \cite{MS} in the one parameter case, ultimately leading to a proof of one side of the Davis
inequality for \eqref{eq:3} filtrations which is to appear in a forthcoming paper \cite{mrz}. As an application
we show that in the biparameter case for a discrete probability space the Hardy space $H^1_S$ is the interpolation 
sum of $H^1_\sigma$ and the space introduced by Garsia in \cite{Gar}
\[
	\calG_1 = \left\{f : \|f\|_{\calG_1} = \sum_{i, j = 1}^\infty \EE \big[|\Delta_{i, j} f|\big] < \infty \right\}.
\]
See Section \ref{sec:2} for details.\par 

It would be tempting to develop a biparameter counterpart of Proposition \ref{1parappr} in order to be able to apply
Theorem \ref{main_thm:2} outside of the discrete case. However, applying the same construction produces a double-indexed
filtration that fails the \eqref{eq:3}-condition in general. As far as we know, the following remains unsolved.
\begin{conj}
	Let $X$ be a separable Banach space. For any $X$-valued martingale $(f_{n,m}:0\leq n\leq N,0\leq m\leq M)$
	on $\left(\Omega,\mathcal{F},\PP\right)$ with respect to a filtration
	$(\calF_{n,m}:0\leq n\leq N,0\leq m\leq M)$ satisfying \eqref{eq:3}-condition, and any $\varepsilon > 0$
	there exists an atomic filtration $(\calG_{n,m}:0\leq n\leq N,0\leq m\leq M)$ satisfying \eqref{eq:3}-condition
	such that
	\[
		\left\|f_{n,m}-\EE\left(f_{N,M}\mid \calG_{n,m}\right)\right\|<\varepsilon
		\quad \text{for all} \quad 0\leq n\leq N,0\leq m\leq M.
 \]
\end{conj}

\subsection*{Notation}
Given an atomic $\sigma$-field $\calF$ by $\at \calF$ we denote the set of atoms of $\calF$. Let $\NN$ denote
the set of positive integers and $\NN_0 = \NN \cup \{0\}$.

\section{One parameter case}
In this section we want to prove Theorem \ref{main_thm:1}. To do so, we construct a sequence of discrete probability
spaces $\big((S, \calS_i, \mu_i) : i \in \NN_0\big)$ and a mapping
\[
	\pi: (\Omega, \calF, \PP) \rightarrow (S, \calS, \mu)
\]
satisfying \eqref{eq:4} and \eqref{eq:5}.

Suppose we are dealing with the simplest nontrivial case, that is $\big(\Omega, \mathcal{B},\mathbb{P}\big)$,
and $\mathcal{A}\subset\mathcal{B}$. Over each atom $A$ of $\mathcal{A}$, there is a different structure of $\mathcal{B}$,
that can be viewed as an individual probability space $(A,\mathcal{B} \cap A,\mathbb{P}_A)$ where
\[
	\PP_A(U) = \frac{\PP(A \cap U)}{\PP(A)},
\]
for any $U \in \calB \cap A$. An atom of $\mathcal{B}$
is in one-to-one correspondence with the choice of an atom of $\mathcal{A}$ and an element of
$(A,\mathcal{B}\cap A,\mathbb{P}_A)$. Moreover, an atom of $\mathcal{B}$ can be recovered if we redundantly include a
choice of an element of $(A',\mathcal{B}\cap A',\mathbb{P}_{A'})$ for all other atoms $A'$ of $\calA$. Next, we identify
an element of
\[
	\bigotimes_{A \in \at \mathcal{A}} (A,\mathcal{B}\cap A,\mathbb{P}_A)
\]
with a mapping $\varphi: \at \calA \rightarrow \at \calB$ such that for an atom $A \in \at \calA$, $\varphi(A)$ is
an atom of $\calB$. That being said, let us define
\begin{align*}
	\pi: (\Omega,\mathcal{A},\mathbb{P}) \otimes 
	\bigotimes_{A \in \at \mathcal{A}} 
	(A,\mathcal{B}\cap A,\mathbb{P}_A) 
	&\longrightarrow 
	(\Omega,\mathcal{B},\mathbb{P}) \\
	(A, \phi) &\longmapsto \varphi(A).
\end{align*}
Now, the condition \eqref{eq:4} is obvious. To check \eqref{eq:5} we consider each $A\in\at \mathcal{A}$ separately.
If $U=\bigcup_{k}B_k$, where $B_k\subset A$ are disjoint atoms of $\mathcal{B}$, then 
\[
	\pi^{-1}(U)= A \times \left\{\varphi:\varphi(A)\subset U\right\}.
\]
By definition of $\varphi$, the second factor is just a condition on the $A$-th coordinate of $\varphi$, and its measure
equals
\[
	\mathbb{P}_A(U)=\frac{\mathbb{P}(U)}{\mathbb{P}(A)}
\]
verifying \eqref{eq:5}.

In a general case, we have a filtration $(\calF_n : 1 \leq n \leq N)$ in a probability space $(\Omega, \calF, \PP)$.
We define a map
\begin{align*}
	\pi:\left(\Omega,\mathcal{F}_0,\mathbb{P} \right) \otimes \bigotimes_{n=1}^N \bigotimes_{A\in \at\mathcal{F}_{n-1}} 
	(A,\mathcal{F}_n\cap A,\mathbb{P}_A) 
	&\longrightarrow 
	\left(\Omega,\mathcal{F}_N,\mathbb{P}\right) \\
	(A, \varphi_1, \ldots, \varphi_N) 
	&\longmapsto
	\varphi_N \circ \varphi_{N-1} \circ \ldots \circ \varphi_1(A)
\end{align*}
where, as previously, an atom of $\bigotimes_{A\in \at\mathcal{F}_{n-1}}\left(A,\mathcal{F}_n\cap A,\mathbb{P}_A\right)$
we treat as a function $\varphi_n:\at\mathcal{F}_{n-1}\to\at\mathcal{F}_n$ satisfying
$\varphi_n(A_{n-1}) \subset A_{n-1}$. 
From definition of $\pi$ it is obvious that for an atom $B_n$ of $\mathcal{F}_n$, the condition
$\pi\left(A,\varphi_1,\ldots,\varphi_N\right) \in B_n$ is equivalent to
$\varphi_n \circ \varphi_{n-1} \circ \ldots \circ \varphi_1(A)=B_n$, so it
depends only on $A$ and $\varphi_i$ for $i \leq n$ proving \eqref{eq:4}. 
The condition \eqref{eq:5} can be check on atoms of $\mathcal{F}_N$. If $A_N\in \at\mathcal{F}_N$, then, denoting its
ancestors by $A_n\in \at\mathcal{F}_n$, we have
\[
	\pi \left(A,\varphi_1,\ldots,\varphi_N\right) =A_N 
\]
if and only if $A = A_0$, and
\[
	\varphi_n(A_{n-1}) = A_n, \qquad\text{for all } 1 \leq n \leq N.
\]
The probability of this event equals
\[
	\mathbb{P}(A_0) \prod_{n=1}^N \frac{\mathbb{P}(A_n)}{\mathbb{P}(A_{n-1})}= \mathbb{P}(A_N)
\]
which concludes the proof of Theorem \ref{main_thm:1}.

\section{Two parameter case}
In this section we prove Theorem \ref{main_thm:2}. In the two parameter case it is convenient to use the following variant
of mathematical induction.
\begin{lemma}
	\label{lem:1}
	Suppose that a set $X \subset \NN^2$ satisfies
	\begin{enumerate}
		\item $(1, 1) \in X$,
		\item if $(i, 1) \in X$, then $(i+1, 1) \in X$,
		\item if $(1, j) \in X$, then $(1, j+1) \in X$,
		\item if $(i+1, j), (i, j+1), (i, j) \in X$, then $(i+1, j+1) \in X$,
	\end{enumerate}
	then $X = \NN^2$.
\end{lemma}
Again, our aim is to construct a double-indexed sequence of probability spaces $(S_{i, j}, \calS_{i, j}, \mu_{i, j})$,
and a mapping
\[
	\pi: \bigotimes_{i = 1}^N \bigotimes_{j = 1}^M (S_{i, j}, \calS_{i, j}, \mu_{i, j})
	\rightarrow
	(\Omega, \calF_{N, M}, \PP)
\]
satisfying \eqref{eq:4} and \eqref{eq:5}. We use similar idea as in the one parameter case. For $1 \leq i \leq N$ and
$1 \leq j \leq M$, we set 
\[
	\big(S_{i, j}, \calS_{i, j}, \mu_{i, j}\big) = \bigotimes_{A \in \at \calF_{i, j}^-} (A, \calF_{i, j} \cap A, \PP_A)
\]
where
\[
	\calF_{i, j}^- = \calF_{i-1, j} \vee \calF_{i, j-1}
\]
with
\begin{equation}
	\label{eq:33}
	\mathcal{F}_{i,0} = \{\emptyset,\Omega\},
	\quad\text{and}\quad
	\mathcal{F}_{0, j} = \{\emptyset,\Omega\}.
\end{equation}
Hence, the atoms in $\calS_{i, j}$ are sequences having a form $(B_A : A \in \at\mathcal{F}_{i,j}^-)$ 
where $B_A$ denotes an atom of $\calF_{i, j}$ contained in $A$. Observe that
for such an atom we have
\[
	\mu_{i, j}\Big( B_A : A \in \at \calF_{i, j}^{-} \Big) = \prod_{A \in \at \calF_{i, j}^-} \PP_A(B_A).
\]
Consequently, atoms of the domain of $\pi$ are
\begin{equation}
	\label{eq:6}
	\mathbf{B}
	=\Big( \big( B^{i,j}_A : A\in \at\mathcal{F}_{i,j}^- \big) : 1 \leq i \leq N, 1 \leq j \leq M \Big).
\end{equation}
We are now ready to define the mapping $\pi$, namely for an atom of the form \eqref{eq:6} we set
\begin{equation}
	\label{eq:10}
	\pi(\mathbf{B})= \bigcap_{i = 1}^N \bigcap_{j = 1}^M \bigcup_{A\in\at\mathcal{F}_{i,j}^-} B^{i,j}_A.
\end{equation}
Observe that the right hand-side of \eqref{eq:10} can be written as a disjoint union of sets of the form
\begin{equation}
	\label{eq:16}
	\bigcap_{i = 1}^N \bigcap_{j = 1}^M B^{i,j}_{A_{i,j}}
\end{equation}
while $A_{i,j}$ runs over all atoms of $\mathcal{F}_{i,j}^-$. Suppose that there is a sequence
$(A_{i, j} : 1 \leq i \leq N, 1 \leq j \leq M)$ so that $A_{i, j} \in \at \calF^-_{i, j}$ and
\[
	\bigcap_{i = 1}^N \bigcap_{j = 1}^M B^{i,j}_{A_{i,j}} \neq \emptyset.
\]
We are going to use the induction procedure given by Lemma \ref{lem:1}. 
For $2 \leq i \leq N$ and $2 \leq j \leq M$, we have
\[
	B^{i-1, j}_{A_{i-1,j}} \cap B^{i, j-1}_{A_{i, j-1}} \cap B^{i, j}_{A_{i, j}} \neq \emptyset.
\]
Since
\[
	B^{i-1, j}_{A_{i-1, j}} \cap B^{i, j-1}_{A_{i, j-1}} \in \calF^-_{i, j} \subset \calF_{i, j}.
\]
we conclude that 
\begin{equation}
	\label{eq:8}
	B^{i-1, j}_{A_{i-1, j}} \cap B^{i, j-1}_{A_{i, j-1}} \supset B^{i, j}_{A_{i, j}}.
\end{equation}
Moreover, $B^{i-1, j}_{A_{i-1, j}} \cap B^{i, j-1}_{A_{i, j-1}}$ is an atom of $\calF_{i, j}^-$, thus
\begin{equation}
	\label{eq:12}
	A_{i, j} = B^{i-1, j}_{A_{i-1, j}} \cap B^{i, j-1}_{A_{i, j-1}}.
\end{equation}
Similarly, for $2 \leq i \leq N$,
\[
	B^{i-1, 1}_{A_{i-1, 1}} \cap B^{i, 1}_{A_{i, 1}} \neq \emptyset
\]
and since $B^{i-1, 1}_{A_{i-1, 1}}$ and $B^{i, 1}_{A_{i, 1}}$ are atoms of $\calF_{i-1, 1}$ and $\calF_{i, 1}$,
respectively, we obtain that
\begin{equation}
	\label{eq:17}
	B^{i-1, 1}_{A_{i-1, 1}} \supset B^{i, 1}_{A_{i, 1}}.
\end{equation}
Because
\[
	\calF_{i-1, 1} = \calF_{i-1, 1} \vee \calF_{i, 0} = \calF^-_{i, 1},
\]
$B^{i-1, 1}_{A_{i-1, 1}}$ is an atom of $\calF^-_{i ,1}$, thus
\begin{equation}
	\label{eq:19}
	A_{i, 1} = B^{i-1, 1}_{A_{i-1,1}}.
\end{equation}
Analogously, for $2 \leq j \leq N$, we conclude
\begin{equation}
	\label{eq:21}
	B^{1, j-1}_{A_{1, j-1}} \supset B^{1, j}_{A_{1, j}},
\end{equation}
and
\begin{equation}
	\label{eq:18}
	A_{1, j} = B^{1, j-1}_{A_{1, j-1}}.
\end{equation}
Finally, $B^{1, 1}_{A_{1, 1}}$ is any atom of $\calF_{1, 1}$, thus
\begin{equation}
	\label{eq:20}
	A_{1, 1} = \Omega.
\end{equation}
Using \eqref{eq:8}, \eqref{eq:17} and \eqref{eq:18}, by Lemma \ref{lem:1}, for $1 \leq n \leq N$ and $1 \leq m \leq M$
we obtain that
\[
	\bigcap_{i = 1}^n \bigcap_{j = 1}^m B^{i, j}_{A_{i, j}} = B^{n, m}_{A_{n, m}}.
\]
In particular,
\begin{equation}
	\label{eq:22}
	\pi(\mathbf{B}) = \bigcap_{i = 1}^N \bigcap_{j = 1}^M B^{i, j}_{A_{i, j}} = B^{N, M}_{A_{N, M}},
\end{equation}
provided that $\pi(\mathbf{B}) \neq \emptyset$.

We are now in the position to verify \eqref{eq:4}. If $U$ is an atom of $\calF_{n, m}$, then
$\mathbf{B} \subset \pi^{-1}(U)$ is equivalent to $B^{n, m}_{A_{n, m}} = U$. In view of
\eqref{eq:12}, \eqref{eq:19}, \eqref{eq:18} and \eqref{eq:20}, $B^{n, m}_{A_{n, m}}$ depends only on $A_{i, j}$
for $1 \leq i \leq n$ and $1 \leq j \leq m$, thus $\pi^{-1}(U)$ belongs to $\calS_{n, m}$. To show
\eqref{eq:5}, it is sufficient to consider $U$ being an atom of $\calF_{N, M}$. Suppose that $\mathbf{B} \subset 
\pi^{-1}(U)$. By \eqref{eq:22}, for each $1 \leq i \leq N$ and $1 \leq j \leq M$, $B^{i, j}_{A_{i, j}}$ is
the unique atom of $\calF_{i, j}$ containing $U$. Therefore,
\[
	\mu\big(\pi^{-1}(U) \big) = 
	\mu\big(\mathbf{B}\big) =
	\prod_{i = 1}^N \prod_{j = 1}^M \PP_{A_{i, j}}\big(B^{i, j}_{A_{i, j}}\big).
\]
Now it is enough to show that for each $1 \leq n \leq N$ and $1 \leq m \leq M$,
\begin{equation}
	\label{eq:23}
	\prod_{j = 1}^n \prod_{j = 1}^m \PP_{A_{i, j}}\big(B^{i, j}_{A_{i, j}}\big) 
	= \PP\big(B^{n, m}_{A_{n, m}} \big).
\end{equation}
For the proof we use the induction procedure given by Lemma \ref{lem:1}. For $n = m = 1$, there is nothing to be proved
since $A_{1, 1} = \Omega$. For $n > 1$ and $m = 1$, by \eqref{eq:19}, we have
\begin{align*}
	\prod_{i = 1}^n \PP_{A_{i, 1}}\big( B^{i, 1}_{A_{i, 1}} \big)
	&=
	\prod_{i = 1}^n \frac{\PP\big(B^{i, 1}_{A_{i, 1}}\big)}{\PP(A_{i, 1})} \\
	&=
	\PP\big(B^{1, 1}_{A_{1, 1}}\big) 
	\prod_{i = 2}^n \frac{\PP\big(B^{i, 1}_{A_{i, 1}}\big)}{\PP\big(B^{i-1, 1}_{A_{i-1, 1}}\big)}
	=
	\PP\big(B_{A_{i, 1}}^{i, 1}\big).
\end{align*}
For $n = 1$ and $m > 1$ the reasoning is analogous. Now, let us suppose that \eqref{eq:23} holds true for $(n-1, m-1)$,
$(n-1, m)$ and $(n, m-1)$ for some $2 \leq n \leq N$ and $2 \leq m \leq M$. Then
\begin{align}
	\nonumber
	\prod_{i = 1}^n \prod_{j = 1}^m \PP_{A_{i, j}} \big(B^{i, j}_{A_{i, j}}\big) 
	&=
	\bigg(\prod_{i = 1}^{n-1} \prod_{j = 1}^m \PP_{A_{i, j}} \big(B^{i, j}_{A_{i, j}}\big)\bigg)
	\bigg(\prod_{i = 1}^n \prod_{j = 1}^{m-1} \PP_{A_{i, j}} \big(B^{i, j}_{A_{i, j}}\big)\bigg) \\
	\nonumber
	&\phantom{=}\times
	\bigg(\prod_{i = 1}^{n-1} \prod_{j = 1}^{m-1} \PP_{A_{i, j}} \big(B^{i, j}_{A_{i, j}}\big)\bigg)^{-1} 
	\PP_{A_{n, m}}\big(B^{n, m}_{A_{n, m}}\big)\\
	\label{eq:24}
	&=
	\frac{\PP\big(B^{n-1, m}_{A_{n-1, m}}\big) \PP\big(B^{n, m-1}_{A_{n, m-1}}\big)}
	{\PP\big(B^{n-1, m-1}_{A_{n-1, m-1}}\big) \PP(A_{n, m})}
	\PP\big(B^{n, m}_{A_{n, m}} \big).
\end{align}
Observe that by the conditional independence
\begin{align*}
	\frac{\PP\big(B^{n-1, m}_{A_{n-1, m}}\big) \PP\big(B^{n, m-1}_{A_{n, m-1}}\big)}
	{\PP\big(B^{n-1, m-1}_{A_{n-1, m-1}}\big) \PP(A_{n, m})}
	&=
	\PP\big(B^{n-1, m}_{A_{n-1, m}} | B^{n-1, m-1}_{A_{n-1, m-1}} \big) 
	\PP\big(B^{n, m-1}_{A_{n, m-1}} | B^{n-1, m-1}_{A_{n-1, m-1}} \big)
	\frac{\PP\big(B^{n-1, m-1}_{A_{n-1, m-1}}\big)}{\PP(A_{n, m})} \\
	&=
	\PP\big(B^{n-1, m}_{A_{n-1, m}} \cap B^{n, m-1}_{A_{n, m-1}} | B^{n-1, m-1}_{A_{n-1, m-1}} \big)
	\frac{\PP\big(B^{n-1, m-1}_{A_{n-1, m-1}}\big)}{\PP(A_{n, m})} \\
	&=
	\frac{\PP\big(B^{n-1, m}_{A_{n-1, m}} \cap B^{n, m-1}_{A_{n, m-1}}\big)}
	{\PP(A_{n, m})} = 1,
\end{align*}
where the last equality is a consequence of \eqref{eq:12}. Therefore, by \eqref{eq:24}, we conclude that
\eqref{eq:23} holds true proving Theorem \ref{main_thm:2}.

\section{Applications to decoupling in martingale $H^1$ spaces}
\label{sec:2}
In this section we demonstrate that canonical product filtrations -- both in one and two parameters -- provide a natural
framework for constructing and dealing with \emph{decoupled sequences}. For a thorough exposition on this concept, we
refer the reader to \cite{GdlP}. We start by recalling the definition of a decoupled tangent sequence.

Let $(\Omega, \calF, \PP)$ be a probability space with a filtration $(\calF_k : 1 \leq k \leq n)$, 
$(\tilde{\Omega}, \tilde{\calF})$ be a measurable space with a filtration $(\tilde{\calF_k} : 1 \leq k \leq n)$, and 
$\tilde{\PP}: \Omega \times \tilde{\calF} \rightarrow [0, 1]$ be a probability transition function. We construct a new
probability space $(\Omega \times \tilde{\Omega}, \calF \otimes \tilde{\calF}, \PP \otimes \tilde{\PP})$, where
\[
	\PP \otimes \tilde{\PP}(A \times B) = \int_A \tilde{\PP}(\omega, B) \ \PP({\rm d} \omega)
\]
for $A \in \calF$ and $B \in \tilde{\calF}$. Given a sequence $(f_k : 1 \leq k \leq n)$ adapted to
$(\calF_k : 1 \leq k \leq n)$, a sequence $(g_k : 1 \leq k \leq n)$ adapted to $(\calF_k \otimes \tilde{\calF_k} :
1 \leq k \leq n)$ is called \emph{decoupled tangent sequence} to $(f_k)$, if for each $\omega \in \Omega$,
$(g_k(\omega, \cdot) : 1 \leq k \leq n)$ is a sequence of independent random variables on $(\tilde{\Omega}, \tilde{\calF},
\tilde{\PP}(\omega, \cdot))$, and
\begin{equation}
	\label{eq:30}
	\EE\big[f_k \otimes \ind{\Omega'} \big| \calF_{k-1} \otimes \tilde{\calF}_{k-1} \big]
	=
	\EE\big[g_k \big| \calF_{k-1} \otimes \tilde{\calF}_{k-1}\big]
\end{equation}
for all $k = 1, \ldots, n$.

The above definition, as well as the proof of existence, may seem to be quite abstract, however, there is an explicit
construction in the case of the canonical filtration on a product space (see \cite[Example 4.3.3]{KW})
\[
	(S, \calS, \mu) = \bigotimes_{j = 1}^n (S_j, \calS_j, \mu_j)
\]
To do so, we take $\tilde{\Omega} = S$, $\tilde{\calF}=\calF$, $\tilde{\calF}_k = \calF_k$, and
$\tilde{\PP}(\omega, B) = \PP(B)$, and define
\begin{equation}
	\label{eq:25}
	g_k(\omega_1, \ldots, \omega_k, \tilde{\omega}_1, \ldots, \tilde{\omega}_k) 
	= 
	f_k(\omega_1, \ldots, \omega_{k-1}, \tilde{\omega}_k).
\end{equation}
In view of a theorem due to Kwapień and Woyczynski \cite[Theorem 5.2.1]{KW}, if $f_k$ and $g_k$ are positive functions
satisfying \eqref{eq:25} then for any increasing concave function $\phi:\mathbb{R}_+\to\mathbb{R}_+$,
\begin{equation}
	\label{eq:26}
	\mathbb{E} \phi\left(f_1+\ldots+f_k\right)\simeq \mathbb{E} \phi\left(g_1+\ldots+g_k\right)
\end{equation}
where the implicit constants in \eqref{eq:26} are absolute. This observation is particularly useful when dealing with
martingale Hardy spaces defined in terms of square functions. To be more precise, let $(\Omega, \calF, \PP)$ be a
probability space with a filtration $(\calF_k : k \in \NN_0)$, $\calF_0 = \{\emptyset, \Omega\}$. For a martingale
$f = (f_k : k \in \NN)$, the square function $S(f)$ and the conditional square function $s(f)$ are defined as
\[
	S(f) = \Big(\sum_{k = 1}^\infty \big|\Delta_k(f) \big|^2\Big)^{\frac{1}{2}}
	\qquad\text{and}\qquad
	s(f) = \Big(\sum_{k = 1}^\infty \EE\big( | \Delta_k(f) |^2 \big| \calF_{k-1}\big) \Big)^{\frac{1}{2}}
\]
where
\[
	\Delta_k(f) = 
	\begin{cases}
		f_k - f_{k-1} & \text{if } k \geq 2, \\
		f_1 &\text{otherwise.}
	\end{cases}
\]
Then the Hardy spaces $H^1_S$ and $H^1_s$ consist of martingales $f$ so that
\[
	\|f\|_{H^1_S} = \EE\big[ S(f) \big] < \infty, \qquad\text{and}\qquad \|f\|_{H^1_s} = \EE\big[ s(f) \big] < \infty,
\]
respectively. In fact, there is $C > 0$ such that for all martingales $f$,
\begin{equation}
	\label{eq:11}
	\|S(f) \|_{L^1} \leq C \|s(f)\|_{L^1}.
\end{equation}
In the case of product filtrations, the proof of \eqref{eq:11} is straightforward. Indeed, let
$(\vphi_k : 1 \leq k \leq n)$ be a sequence on $S$ adapted to the canonical filtration
$(\calF_k : 1 \leq k \leq n)$. Then, by \eqref{eq:30} and \eqref{eq:26}, we obtain
\begin{align*}
	\int_S
	\left(\sum_{k = 1}^n \varphi_k\left(x_1,\ldots,x_k\right)\right)^\frac12 \mu({\rm d} x)
	&\leq
	C\int_{S\otimes S}
	\left(\sum_{k=1}^n \varphi_k\left(x_1,\ldots,x_{k-1},y_k\right)\right)^\frac12 \mu\otimes\mu({\rm d}(x, y))\\
	&\leq 
	\int_S \left(\int_S \sum_{k=1}^n \varphi_k\left(x_1,\ldots,x_{k-1},y_k\right) \mu({\rm d} y) 
	\right)^\frac12 \mu({\rm d} x)\\
	&= 
	\int_S
	\left(\sum_{k=1}^n \int_{S_k} \varphi_k\left(x_1,\ldots,x_{k-1},t\right) \mu_k({\rm d } t) \right)^\frac12 
	\mu({\rm d} x).
\end{align*}
Now, to conclude \eqref{eq:11}, it is enough to take $\vphi_k = |\Delta_k f|^2$.

Now, let us turn to biparameter case. For a martingale $f = (f_{i, j} : i, j \in \NN)$, the martingale differences
are defined as
\[
	\Delta_{i, j}(f) = 
	\begin{cases}
		f_{i, j} - f_{i-1, j} - f_{i, j -1} + f_{i-1, j-1} & \text{if } i \geq 2, j \geq 2, \\
		f_{1, j} - f_{1, j-1} &\text{if } i = 1, j \geq 2, \\
		f_{i, 1} - f_{i-1, 1} &\text{if } i \geq 2, j = 1, \\
		f_{1, 1} & \text{if } i = j = 1.
	\end{cases}
\]
In this context, there are three square functions
\[
	S(f) = \Big(\sum_{i = 1}^\infty \sum_{j = 1}^\infty |\Delta_{i, j} (f)|^2\Big)^{\frac{1}{2}},
	\qquad
	s(f) = \Big(
	\sum_{i = 1}^\infty \sum_{j = 2}^\infty \EE\big[ |\Delta_{i, j} (f)|^2 \big|\calF_{i-1, j-1} \big]
	\Big)^\frac{1}{2},
\]
and
\[
	\sigma(f) = 
	\Big(
	\sum_{i = 1}^\infty \sum_{j = 2}^\infty \EE\big[ |\Delta_{i, j} (f)|^2 \big|\calF^-_{i, j} \big]
	\Big)^\frac{1}{2}
\]
with the convention that $\calF_{i, 0} = \calF_{0, j} = \{\emptyset, \Omega\}$. 
Since $\Delta_{i, j} = \Delta^1_i \Delta^2_j$ where $\Delta^1_i$ and $\Delta^2_j$ are the martingale differences
with respect to
\[
	\calF_{i, \infty} = \sigma\Big(\bigcup_{j = 1}^\infty \calF_{i, j}\Big),
	\qquad\text{and}\qquad
	\calF_{\infty, j} = \sigma\Big(\bigcup_{i = 1}^\infty \calF_{i, j}\Big),
\]
respectively, the estimate \eqref{eq:11} is a simple iteration of one-parameter argument. In fact, we have
the following theorem.
\begin{theorem}
	\label{thm:1}
	Let $(\Omega, \calF, \PP)$ be a discrete probability space with biparameter \eqref{eq:3}-filtration. There
	is a constant $C > 0$ such that for all martingales $f$,
	\begin{equation}
		\label{eq:31}
		\|S(f)\|_{L^1} \leq C \|\sigma(f)\|_{L^1}.
	\end{equation}
\end{theorem}
\begin{proof}
	To prove \eqref{eq:31}, it is sufficient to consider a function $f \in L^0(\calF_{N, N})$ for some $N \in \NN$. 
	Now, in view of Theorem \ref{main_thm:2}, we may assume that we are dealing with a product space
	\[
		(S, \calS, \mu) = \bigotimes_{i = 1}^N \bigotimes_{j = 1}^N (S_{i,j}, \calS_{i, j}, \mu_{i, j})
	\]
	with the canonical biparameter filtration. For $x \in S$ and $A \subseteq [1, N]^2$, we set
	\[
		x_A = \big(x_{i, j} : (i, j) \in A \big).
	\]
	If $x$ and $y$ are two sequences indexed by $A$ and $B$, $A \cap B = \emptyset$, then $(x, y)$ is a sequence indexed
	by $A \cup B$. The range of the operator $\Delta_{i, j}$ is the set of $\calF_{i, j}$-measurable
	functions that are orthogonal to all functions that are $\calF_{i, j-1}$ or $\calF_{i-1, j}$ measurable. Therefore,
	$\Delta_{i, j} (f)$ is a function that depends only on $x_{[1, i] \times [1, j]}$, and
	\[
		\int_{\bigotimes_{k = 1}^i S_{k, j}} \Delta_{i, j}(f) 
		\left(x_{[1, i] \times [1, j-1]}, s_{[1, i] \times \{j\}}\right) 
		\mu_{1, j} \otimes \ldots \otimes \mu_{i, j} ({\rm d} s) = 0,
	\]
	and
	\[
		\int_{\bigotimes_{k = 1}^j S_{i, k}} \Delta_{i, j}(f)
		\left(x_{[1, i-1] \times [1, j]}, t_{\{i\} \times [1, j]}\right)
		\mu_{i, 1} \otimes \ldots \otimes \mu_{i, j} ({\rm d} t) = 0.
	\]
	for all $x \in S$. Moreover, we have
	\[
		\EE\big[|\Delta_{i, j}(f)|^2 \big| \calF_{i, j}^-\big]\left(x_{[1, i] \times [1, j] \setminus \{(i, j)}\right)
		=
		\int_{S_{i, j}} \left|
		\Delta_{i, j}(f)\left(x_{[1, i] \times [1, j] \setminus \{(i, j)\}}, u_{\{(i, j)\}}\right)
		\right|^2
		\mu_{i, j}({\rm d} u).
	\]
	Let $(h_{i,j} : 1 \leq i \leq N, 1 \leq j \leq N)$ be a sequence of functions adapted to the
	canonical filtration. By iterating the inequality \eqref{eq:26} (to be more precise, first decoupling with respect to
	rows of the matrix $x_{[1, i] \times [1, j]}$ and then with respect to the columns of
	$(x_{[1, i-1] \times [1, j]}, y_{\{i\} \times [1, j]})$), we get
	\begin{align*}
		&\int_S
		\left(\sum_{i = 1}^N \sum_{j = 1}^N h_{i, j}\left(x_{[1, i] \times [1, j]}\right)
		\right)^{\frac{1}{2}}
		\mu({\rm d} x) \\
		&\simeq
		\int_{S \otimes S}
		\left(
		\sum_{i = 1}^N \sum_{j = 1}^N h_{i, j}
		\left(x_{[1, i-1] \times [1, j]}, y_{\{i\} \times [1, j]} \right)
		\right)^{\frac{1}{2}}
		\mu \otimes \mu ({\rm d} (x, y)) \\
		&\simeq
		\int_{S^{\otimes 4}}
		\left(
		\sum_{i = 1}^N \sum_{j = 1}^N h_{i, j}\left(
		x_{[1, i-1] \times [1, j-1]}, y_{\{i\} \times [1, j-1]}, z_{[1, i-1] \times \{j\}}, t_{(i, j)} 
		\right)
		\right)^{\frac{1}{2}}
		\mu^{\otimes 4}({\rm d}(x, y, z, t)).
	\end{align*}
	Similarly, we obtain
	\begin{align*}
		&
		\int_S \left( \sum_{i = 1}^N \sum_{j = 1}^N h_{i, j}\left(x_{[1, i]\times[1, j]}\right)\right)^{\frac{1}{2}}
		\mu({\rm d} x) \\
		&\qquad\qquad
		\simeq
		\int_{S \otimes S}
		\left(
		\sum_{i = 1}^N \sum_{j = 1}^N h_{i, j}\left(
		x_{[1, i] \times [1, j] \setminus \{(i, j)\}}, t_{(i, j)}
		\right)
		\right)^{\frac{1}{2}}
		\mu\otimes\mu({\rm d} (x, t)) \\
		&\qquad\qquad\leq
		\int_S
		\left(
		\sum_{i = 1}^N \sum_{j = 1}^N
		\int_{S_{i,j}} 
		h_{i, j}\left(
		x_{[1, i] \times [1, j] \setminus \{(i, j)\}}, t
		\right)
		\mu_{i, j}({\rm d} t)
		\right)^{\frac{1}{2}}
		\mu({\rm d} x)
	\end{align*}
	where in the last step we have used Jensen's inequality. Consequently, by taking $h_{i, j} = |\Delta_{i, j}(f)|^2$,
	we obtain \eqref{eq:31}.
\end{proof}

\begin{remark}
	Theorem \ref{thm:1} is generalization of \cite[Proposition 3.50]{weisz1}, where it was assumed that
	$f$ is a strong martingale, that is
	\[
		\EE\big[\Delta_{i, j}(f) \big| \calF^-_{i, j} \big] = 0,
		\qquad\text{for all } i, j \geq 1.
	\]
	However, in view of \cite[Theorem 3.49]{weisz1}, in this case the inequality \eqref{eq:31} can be easily deduced from
	one parameter case.
\end{remark}
Actually, in the one parameter setting, a stronger inequality than \eqref{eq:11} is true, namely the following
characterization due to Davis and Garsia
\[
	\|f\|_{H^1_S}=\inf_{f=g+h} \|g\|_{H^1_s} + \sum_{k=1}^\infty \EE\big[ |\Delta_k h|\big].
\]
We prove the following theorem.
\begin{theorem}
	\label{thm:2}
	Suppose that $(\Omega, \calF, \PP)$ is a discrete probability space with biparameter \eqref{eq:3}-filtration.
	Then for each martingale $f$,
	\[
		\|f\|_{H^1_S} \simeq \inf_{f = g+h} \|g\|_{H^1_\sigma} + 
		\sum_{i = 1}^\infty \sum_{j = 1}^\infty \EE \big[|\Delta_{i,j}(h)|\big].
	\]
\end{theorem}
\begin{proof}
	Since
	\[
		\EE\Big(\sum_{i = 1}^\infty \sum_{j = 1}^\infty |\Delta_{i,j}(h)|^2 \Big)^\frac{1}{2} 
		\leq
		\EE\Big(\sum_{i = 1}^\infty \sum_{j = 1}^\infty |\Delta_{i,j}(h)|\Big),
	\]
	by Theorem \ref{thm:1}, we get
	\[
		\|g+h\|_{H^1_S} \leq \|g\|_{H^1_S} + \|h\|_{H^1_S}
		\leq C \|g\|_{H^1_\sigma} + \sum_{i = 1}^\infty \sum_{j = 1}^\infty \EE \big[|\Delta_{i,j}(h)|\big].
	\]
	To prove the opposite inequality, it is enough to show that there is $C > 0$ such that for any adapted sequence
	$(\eta_{i, j} : 1 \leq i, j\leq N)$,
	there are adapted sequences $(\phi_{i,j} : 1 \leq i, j \leq N)$ and $(\psi_{i,j} : 1 \leq i, j \leq N)$, so that
	$\eta_{i, j} = \phi_{i, j} + \psi_{i, j}$, and
	\begin{equation}
		\label{eq:32}
		C^{-1} \EE\left(\sum_{i = 1}^N \sum_{j =1}^N | \eta_{i, j} | ^2 \right)^{\frac{1}{2}}
		\geq
		\EE\left(
		\sum_{i = 1}^N \sum_{j = 1}^N \EE\left(|\phi_{i ,j}|^2 \big| \calF_{i, j}^-\right)
		\right)^{\frac{1}{2}}
		+
		\sum_{i = 1}^N \sum_{j =1}^N \EE \big[|\psi_{i, j}|\big].
	\end{equation}
	Indeed, since \eqref{eq:32} implies that
	\[
		C^{-1} \EE\left(\sum_{i = 1}^N \sum_{j =1}^N | \eta_{i, j} | ^2 \right)^{\frac{1}{2}}
		\geq
		\EE\left(
		\sum_{i = 1}^N \sum_{j = 1}^N \EE\left(|D_{i, j}( \phi_{i ,j}) |^2 \big| \calF_{i, j}^-\right)
		\right)^{\frac{1}{2}}
		+
		\sum_{i = 1}^N \sum_{j =1}^N \EE \big[ |D_{i, j} \psi_{i, j}|\big].
	\]
	where
	\[
		D_{i, j}(f) = \EE[f | \calF_{i, j}] - \EE[f | \calF_{i-1, j}] - \EE[f | \calF_{i, j-1}] 
		+ \EE[f | \calF_{i-1, j-1}],
	\]
	we can take $\eta_{i, j} = \Delta_{i, j}(f)$, and
	\[
		g_{i, j} = \sum_{m = 1}^i \sum_{n = 1}^j D_{m, n} (\phi_{m, n}), \qquad
		h_{i, j} = \sum_{m = 1}^i \sum_{n = 1}^j D_{m, n} (\psi_{m, n}).
	\]
	In view of Theorem \ref{main_thm:2}, to show \eqref{eq:32}, we may assume that we are dealing with a product
	space $(S, \calS, \mu)$ equipped with the canonical biparameter filtration. By \eqref{eq:26}, we have
	\[
		\EE\left(\sum_{i = 1}^N \sum_{j =1}^N | \eta_{i, j} | ^2 \right)^{\frac{1}{2}}
		\simeq
		\int_{S \otimes S}
		\left(
		\sum_{i = 1}^N \sum_{j = 1}^N 
		\left| \eta_{i, j}\left(x_{[1, i] \times [1, j] \setminus \{(i ,j)\}}, y_{(i, j)}\right) \right|^2
		\right)^{\frac{1}{2}}
		\mu \otimes \mu ({\rm d} (x, y)).
	\]
	Now, given $x \in S$, the functions 
	\[
		S \ni y \mapsto \eta_{i, j}\left(x_{[1, i] \times [1, j] \setminus \{(i, j)\}}, y_{(i, j)}\right)
	\]
	are independent, thus by \cite[Theorem 1]{JSch}, for $L^\frac{1}{2}$, there are functions
	$(a_{i, j} : 1 \leq i, j \leq N)$ and $(b_{i, j} : 1 \leq i, j \leq N)$, such that
	\begin{align*}
		&
		\int_{S \otimes S}
		\left(
		\sum_{i = 1}^N \sum_{j = 1}^N
		\left| \eta_{i, j}\left(x_{[1, i] \times [1, j] \setminus \{(i ,j)\}}, y_{(i, j)}\right) \right|^2
		\right)^{\frac{1}{2}}
		\mu \otimes \mu ({\rm d} (x, y)) \\
		&\geq
		C
		\int_S
		\left\{
		\left(
		\sum_{i = 1}^N \sum_{j =1}^N
		\int_{S_{i, j}}
		|a_{i, j}(x, y)|^2
		\mu_{i, j}({\rm d} y)
		\right)^{\frac{1}{2}}
		+
		\sum_{i = 1}^N \sum_{j = 1}^N 
		\int_{S_{i, j}}
		|b_{i, j}(x, y)|
		\mu_{i, j}({\rm d} y)
		\right\}
		\mu({\rm d} x),
	\end{align*}
	and $\eta_{i, j} = a_{i, j} + b_{i, j}$. In fact, $a_{i, j}$ and $b_{i, j}$ have disjoint supports.
	This would be the desired decomposition but $a_{i, j}(\:\cdot\:, y)$ and $b_{i, j}(\:\cdot\:, y)$ are not necessarily
	$\calF_{i, j}^-$-measurable. To fix this, we use the following lemma.
	\begin{lemma}
		\label{lem:2}
		Let $(\Omega, \calF, \PP)$ be a finite probability space, $(\calF_\alpha : \alpha \in A)$ be a filtration on
		$\Omega$ indexed by a finite set $A$. Assume that there is $\delta > 0$, so that 
		\begin{equation}
			\label{eq:35}
			\delta \Big\| \sup_{\alpha \in A} \big|\EE\big[f \big| \calF_\alpha\big]\big|\Big\|_{L^2} \leq \|f\|_{L^2}
		\end{equation}
		for all $f \in L^2(\Omega)$. Let $B$ be a finite set. Then for all $\kappa > 0$, and any sequences
		$(w_{\alpha, \beta} : (\alpha, \beta) \in A \times B)$ and $(f_{\alpha, \beta} : (\alpha, \beta) \in A \times B)$,
		such that $f_{\alpha, \beta}$ is $\calF_\alpha$-measurable and $w_{\alpha, \beta} : \Omega \rightarrow [0, 1]$,
		\begin{equation}
			\label{eq:14}
			\EE \left(\sum_{(\alpha, \beta) \in A \times B} |w_{\alpha, \beta} f_{\alpha, \beta}|^2\right)^\frac{1}{2}
			\geq
			\kappa^2 \delta \cdot
			\EE\left(\sum_{(\alpha, \beta) \in A \times B} 
			\ind{A_{\alpha, \beta}^\kappa} |f_{\alpha, \beta}|^2
			\right)^\frac{1}{2}
		\end{equation}
		where
		\[
			A_{\alpha, \beta}^\kappa = 
			\left\{
			\EE[w_{\alpha, \beta} | \calF_\alpha] \geq \kappa
			\right\}.
		\]
	\end{lemma}

	Before we apply Lemma \ref{lem:2}, let us check the following claim.
	\begin{claim}
		\label{clm:1}
		There is $\delta > 0$ such that for all martingales $f$,
		\[
			\delta \Big\|\sup_{(i, j) \in \NN^2} \big|\EE\big[f\big|\calF^-_{i, j} \big]\big|\Big\|_{L^2} 
			\leq \|f\|_{L^2}.
		\]
	\end{claim}
	Since
	\[
		| \EE[f | \calF^-_{i, j}] |
		\leq
		| \EE[f | \calF_{i, j} ] | + 
		\Big(\sum_{i = 1}^\infty \sum_{j = 1}^\infty \big|\EE[f | \calF_{i ,j}] -
		\EE[f | \calF^-_{i, j}] \big|^2\Big)^{\frac{1}{2}},
	\]
	we have
	\begin{align*}
		\Big\|\sup_{(i, j) \in \NN^2} \big|\EE\big[f\big|\calF^-_{i, j} \big]\big|\Big\|_{L^2}
		&\leq
		\Big\|\sup_{(i, j) \in \NN^2} \big|\EE\big[f\big|\calF_{i, j} \big]\big|\Big\|_{L^2} \\
		&\phantom{\leq}+
		\Big\|\Big(\sum_{i = 1}^\infty \sum_{j = 1}^\infty 
		\big|\EE\big[f \big| \calF_{i, j} \big] 
		- \EE\big[f\big|\calF^-_{i, j} \big]\big|^2\Big)^{\frac{1}{2}} \Big\|_{L^2}.
	\end{align*}
	By iterating one-parameter Doob's inequality, we can estimate the first term by a constant multiple of $\|f\|_{L^2}$.
	To bound the second term, we observe that
	\[
		\EE\big[f \big| \calF_{i, j} \big] - \EE\big[f\big|\calF^-_{i, j} \big]
		=
		\Delta_{i, j}(f) - \EE\big[\Delta_{i, j}(f) \big|\calF^-_{i, j} \big],
	\]
	thus
	\[
		\big\|
		\EE\big[f \big| \calF_{i, j} \big] - \EE\big[f\big|\calF^-_{i, j} \big]
		\big\|_{L^2}
		\leq
		\big\|
		\Delta_{i, j}(f)
		\big\|_{L^2},
	\]
	and consequently
	\begin{align*}
		\Big\|\Big(\sum_{i = 1}^\infty \sum_{j = 1}^\infty
		\big|\EE\big[f \big| \calF_{i, j} \big] 
		- \EE\big[f\big|\calF^-_{i, j} \big]\big|^2\Big)^{\frac{1}{2}} \Big\|_{L^2}
		&=
		\Big(\sum_{i = 1}^\infty \sum_{j = 1}^\infty
		\big\| \EE\big[f \big| \calF_{i, j} \big] - \EE\big[f\big|\calF^-_{i, j} \big]
	    \big\|_{L^2}^2
		\Big)^{\frac{1}{2}} \\
		&\leq
		\Big(\sum_{i = 1}^\infty \sum_{j = 1}^\infty
		\big\|\Delta_{i, j}(f)\big\|_{L^2}^2 \Big)^\frac{1}{2} = \|f\|_{L^2},
	\end{align*}
	proving the claim.

	Now, let us see how to apply Lemma \ref{lem:2} to conclude the proof of the theorem. Since $a_{i, j}(x, y) 
	= b_{i, j}(x, y)=0$ if
	$\eta_{i, j}\left(x_{[1, i] \times [1, j] \setminus \{(i, j)\}}, y\right) = 0$, we set
	\[
		w_{(i, j), y}(x) = 
		\frac{a_{i, j}(x, y)}{\eta_{i, j}\left(x_{[1, i]\times[1, j]\setminus\{(i, j)\}}, y\right)}.
	\]
	The spaces $S_{i,j}$ are discrete, thus by Lemma \ref{lem:2}, we obtain
	\begin{align*}
		&
		\int_S
		\left(\sum_{i = 1}^N \sum_{j = 1}^N \int_{S_{i, j}} |a_{i, j}(x, y)|^2 
		\mu_{i, j}({\rm d} y) \right)^{\frac{1}{2}}
		\mu({\rm d} x) \\
		&\qquad\qquad=
		\int_S
		\left(
		\sum_{i = 1}^N \sum_{j = 1}^N \int_{S_{i,j}}
		\left|
		w_{(i, j), y}(x) \eta_{i, j}\left(x_{[1, i] \times [1, j] \setminus \{(i, j)\}}, y\right)
		\right|^2
		\mu_{i,j}({\rm d} y)
		\right)^{\frac{1}{2}}
		\mu({\rm d} x) \\
		&\qquad\qquad\gtrsim
		\int_S
		\left(
		\sum_{i = 1}^N \sum_{j =1}^N \int_{S_{i, j}}
		\left|
		\eta_{i, j}\left(x_{[1, i] \times [1, j] \setminus \{(i, j)\}}, y \right)
		\ind{A_{(i, j), y}}(x)
		\right|^2
		\mu_{i, j}({\rm d} y)
		\right)^{\frac{1}{2}}
		\mu({\rm d} x)
	\end{align*}
	where
	\[
		A_{(i, j), y} = \left\{x \in S : 
		\EE\big[w_{(i, j), y} \big| \calF^-_{i, j}\big](x)
		\geq \frac{1}{2} \right\}.
	\]
	Since
	\begin{align*}
		&
		\int_S \big(1 - w_{(i, j), y}(x)\big) 
		\left|\eta_{i, j}\left(x_{[1, i] \times [1, j] \setminus \{(i, j)\}}, y\right) \right|
		\mu({\rm d} x) \\
		&\qquad\qquad=
		\int_S \EE\big[1 - w_{(i, j), y} \big| \calF_{i, j}^-\big] 
		\left|\eta_{i, j}\left(x_{[1, i] \times [1, j] \setminus \{(i, j)\}}, y\right) \right|
		\mu({\rm d} x) \\
		&\qquad\qquad\geq
		\frac{1}{2}
		\int_S 
		\ind{A^c_{(i, j), y}}(x) 
		\left|\eta_{i, j}\left(x_{[1, i] \times [1, j] \setminus \{(i, j)\}}, y\right) \right|
		\mu({\rm d} x),
	\end{align*}
	we obtain
	\begin{align*}
		&
		\int_S \sum_{i = 1}^N \sum_{j = 1}^N \int_{S_{i, j}} |b_{i, j}(x, y)|^2 \mu_{i, j}({\rm d} y)
		\mu({\rm d} x)\\
		&\qquad\qquad=
		\sum_{i = 1}^N \sum_{j = 1}^N \int_{S_{i, j}}
		\int_S
		\big(1 - w_{(i, j), y}(x)\big)
		\left|\eta_{i, j}\left(x_{[1, i] \times [1, j] \setminus \{(i, j)\}}, y\right) \right|
		\mu({\rm d} x)
		\mu_{i, j}({\rm d} y) \\
		&\qquad\qquad\gtrsim
		\sum_{i = 1}^N \sum_{j = 1}^N \int_{S_{i, j}} \int_S
		\left|\eta_{i, j}\left(x_{[1, i] \times [1, j] \setminus \{(i, j)\}}, y\right) \right|
		\ind{A^c_{(i, j), y}}(x)
		\mu({\rm d} x)
		\mu_{i, j}({\rm d} y) \\
		&\qquad\qquad=
		\sum_{i = 1}^N \sum_{j = 1}^N
		\int_S
		\left|\eta_{i, j}\left(x_{[1, i] \times [1, j]}\right)\right|
		\ind{A^c{(i, j), x_{(i, j)}}}
		\left(x_{[1, i] \times [1, j] \setminus \{(i, j)\}}\right)
		\mu({\rm d} x).
	\end{align*}
	Therefore, setting
	\[
		\phi_{i, j}\left(x_{[1, i] \times [1, j]}\right) 
		=
		\eta_{i, j}\left(x_{[1, i] \times [1, j]}\right) \ind{A_{(i, j), x_{(i, j)}}}
		\left(x_{[1, i] \times [1, j]\setminus\{(i, j)\}}\right)
	\]
	and
	\[
		 \psi_{i, j}\left(x_{[1, i] \times [1, j]}\right)
		 =
		 \eta_{i, j}\left(x_{[1, i] \times [1, j]}\right) \ind{A^c_{(i, j), x_{(i, j)}}}
		\left(x_{[1, i] \times [1, j]\setminus\{(i, j)\}}\right),
	\]
	we conclude the proof of \eqref{eq:32}, and hence the theorem follows.

	It remains now to prove Lemma \ref{lem:2}.
	\begin{proof}[Proof of Lemma \ref{lem:2}]
	Let us recall that $\Omega$ as well as $A$ and $B$ are finite sets. For a double-indexed sequence 
	$(f_{\alpha, \beta} : (\alpha, \beta) \in A \times B)$, we set
	\[
		f_\alpha = (f_{\alpha, \beta} : \beta \in B) 
		\qquad\text{and}\qquad
		f = (f_\alpha : \alpha \in A).
	\]
	By $W_\alpha$ and $W$ we denote multiplication operators acting on $\ell^2(B)$ and $\ell^2(A, \ell^2(B))$ as
	\[
		W_\alpha\big(f_\beta : \beta \in B \big) = \big(w_{\alpha, \beta} f_\beta : \beta \in B\big),
	\]
	and
	\[
		W \big(f_\alpha : \alpha \in A\big) = \big(W_\alpha f_\alpha : \alpha \in A\big),
	\]
	respectively. Let
	\[
		\calA = \left\{
		\phi: \Omega \rightarrow \ell^2\big(A \times B\big) : 
		\begin{aligned}
			\phi_{\alpha, \beta} \text{ is } \calF_{\alpha}\text{-measurable, and } \\
			\supp \phi_{\alpha, \beta} \subseteq A_{\alpha, \beta}^\kappa 
			\text{ for all } (\alpha, \beta) \in A \times B
		\end{aligned}
		\right\}.
	\]
	For $p \in [1, \infty)$ we set
	\[
		Y_p = \left\{\phi \in L^p\big(\ell^2(A \times B) \big) : \phi \in \calA \right\}.
	\]
	We denote by $Y_p^*$ the Banach space dual to $Y_p$ with respect to the pairing
	\begin{align*}
		\sprod{\phi}{\psi} 
		&= 
		\sum_{\alpha \in A} \EE \sprod{\phi_\alpha}{\psi_\alpha}_{\ell^2(B)} \\
		&=
		\sum_{(\alpha, \beta) \in A \times B} 
		\EE \big(\phi_{\alpha, \beta} \psi_{\alpha, \beta}\big).
	\end{align*}
	Without loss of generality, we can assume that $f \in \calA$. Indeed, if we replace $f_{\alpha,\beta}$
	by $f_{\alpha, \beta} \ind{A_{\alpha, \beta}^\kappa}$ then the left hand-side of \eqref{eq:14} is decreased. 
	Since 
	\begin{align}
		\nonumber
		\EE\Big(\sum_{\alpha \in A} \|W_\alpha f_\alpha \|_{\ell^2(B)}^2 \Big)^\frac{1}{2}
		&\geq
		\Big(
		\sum_{(\alpha, \beta) \in A \times B}
		\EE\big[|w_{\alpha, \beta}|\cdot |f_{\alpha, \beta}| \big]^2
		\Big)^{\frac{1}{2}} \\
		\label{eq:36}
		&\geq
		\kappa
		\Big(\sum_{(\alpha, \beta) \in A \times B} \EE[|f_{\alpha, \beta}|]^2\Big)^\frac{1}{2},
	\end{align}
	if the left-hand side of \eqref{eq:14} equals zero, then $f \equiv 0$. Therefore, we can assume that
	$\|W f\|_{\ell^2(A \times B)}$ and $f$ are not identically zero.

	Observe that $Y_1$ is finite dimensional, thus to prove the lemma it is enough to show that for any $g \in Y_1^*$
	with $\|g\|_{Y_1^*} = 1$,
	\[
		\EE\Big(\sum_{\alpha \in A} \|W_\alpha f_\alpha \|_{\ell^2(B)}^2 \Big)^\frac{1}{2}
		\geq \kappa^2 \delta \sprod{f}{g}.
	\]
	Let us fix $g$ on the unit sphere in $Y_1^*$. We can assume that $\sprod{f}{g} = 1$. We claim that the function
	\[
		\begin{alignedat}{1}
			\Phi: V &\longrightarrow [0, \infty) \\
			f &\longmapsto \EE\Big(\sum_{\alpha \in A} \|W_\alpha f_\alpha \|_{\ell^2(B)}^2 \Big)^\frac{1}{2}
		\end{alignedat}
	\]
	where $V = \{f \in Y_1 : \sprod{f}{g} = 1\}$, attains its minimum. First, let us notice that for $\lambda \in [0, 1]$
	and $\phi, \psi \in V$, we have $\lambda \phi + (1- \lambda) \psi \in V$, and
	\begin{align*}
		\Phi(\lambda \phi + (1-\lambda)\psi)
		&\leq
		\Phi(\lambda \phi) + \Phi((1-\lambda)\psi) \\
		&=
		\lambda\Phi(\phi) + (1-\lambda) \Phi(\psi),
	\end{align*}
	that is $\Phi$ is convex. Moreover, since norms in a finite dimensional space are comparable, by \eqref{eq:36}
	we get
	\[
		\Phi(f) \gtrsim \|f\|_{Y_1}.
	\]
	Consequently, the function $\Phi$ attains its minimum on $V$. Let $f$ be the minimizer. By
	the Lagrange multiplier method, there is $\lambda \in \RR$ so that 
	\begin{align*}
		\lambda = \lambda \sprod{g}{f} &= 
		\left. \frac{{\rm d} }{{\rm d} t} \Phi(f + t f) \right|_{t = 0} \\
		&=
		\left. \frac{{\rm d} }{{\rm d} t} (1+t) \Phi(f) \right|_{t = 0} = \Phi(f).
	\end{align*}
	Thus, it remains to show that $\lambda \geq \kappa^2 \delta$. To do so, we consider a Banach subspace
	$Z_1 = \{u \in Y_1 : u \in \calB\}$ where
	\[
		\calB = \left\{u \in \calA : \supp u_{\alpha, \beta} \subseteq \supp f_{\alpha, \beta}
		\text{ for all } (\alpha, \beta) \in A \times B \right\}.
	\]
	Let us notice that for $u \in \calA$, we have $(\pi_\alpha u_\alpha : \alpha \in A) \in \calB$ where
	\[
		\pi_\alpha u_\alpha = \left(u_{\alpha, \beta} \ind{\supp f_{\alpha, \beta}} : \beta \in B \right).
	\]
	Since the support of $\|W(f+t u)\|_{\ell^2(A \times B)}$, for $t \in \RR$ and $u \in Z_1$, is contained in
	\[
		U = \big\{\|W f\|_{\ell^2(A \times B)} > 0 \big\},
	\]
	we obtain
	\begin{align*}
		\nonumber
		\lambda \sprod{g}{u} = 
		\left.
		\frac{{\rm d}}{{\rm d} t}
		\Phi(f + t u)
		\right|_{t = 0}
		&=
		\sum_{(\alpha, \beta) \in A \times B}
		\EE\left[\Big(\sum_{\alpha' \in A} \|W_{\alpha'} f_{\alpha'}\|_{\ell^2(B)}^2\Big)^{-\frac{1}{2}}
		\ind{U}
		w_{\alpha, \beta}^2 f_{\alpha,\beta} u_{\alpha, \beta}
		\right] \\
		&=
		\sum_{\alpha \in A}
		\EE
		\left\langle\EE\big[\|W f\|_{\ell^2(A \times B)}^{-1}\ind{U} 
		W_\alpha^2 f_{\alpha} \big| \calF_\alpha \big], u_\alpha
		\right\rangle_{\ell^2(B)}.
	\end{align*}
	Hence,
	\begin{align*}
		\lambda 
		= \|\lambda g \|_{Y_1^*}
		\geq
		\|\lambda g\|_{Z_1^*}
		=
		\left\|
		\left( 
		\EE\big[\|W f\|_{\ell^2(A \times B)}^{-1} \ind{U}
		W_{\alpha}^2 f_{\alpha} \big| \calF_\alpha \big] : \alpha \in A\right)
		\right\|_{Z_1^*}.
	\end{align*}
	Next, we write
	\begin{align*}
		&\left\|
		\left(
		\EE\big[\|W f\|_{\ell^2(A \times B)}^{-1} \ind{U}
		W_{\alpha}^2 f_{\alpha} \big| \calF_\alpha \big] : \alpha \in A\right)
		\right\|_{Z_1^*} \\
		&\qquad\qquad=
		\sup_{u \in \calB}
		\frac{\left| 
		\sum_{\alpha \in A} \EE
		\Big\langle
		\EE\big[ \|W f\|_{\ell^2(A \times B)}^{-1} \ind{U} W_\alpha^2 f_{\alpha} \big| \calF_\alpha \big], 
		u_\alpha \Big\rangle_{\ell^2(B)}
		\right|}{\|u\|_{Y_1}} \\
		&\qquad\qquad
		=
		\sup_{u \in \calA}
		\frac{\left| 
		\sum_{\alpha \in A} \EE
		\Big\langle
		\EE\big[ \|W f\|_{\ell^2(A \times B)}^{-1} \ind{U} W_\alpha^2 f_{\alpha} \big| \calF_\alpha \big], 
		u_\alpha \Big\rangle_{\ell^2(B)}
		\right|}{\EE\big(\|\pi_\alpha u_\alpha\|_{\ell^2(B)}^2 \big)^{\frac{1}{2}}}.
	\end{align*}
	Since $\pi_\alpha u_\alpha(\omega) \neq 0$ implies that $f_{\alpha, \beta}(\omega) \neq 0$ for some $\beta \in B$,
	thus $\EE[w_{\alpha, \beta} | \calF_\alpha](\omega) > \kappa$, which leads to 
	$\EE[\|Wf \|_{\ell^2(A \times B)} | \calF_\alpha](\omega) > 0$. Therefore,
	\begin{align*}
		&\left\|
		\left(
		\EE\big[\|W f\|_{\ell^2(A \times B)}^{-1} \ind{U}
		W_{\alpha}^2 f_{\alpha} \big| \calF_\alpha \big] : \alpha \in A\right)
		\right\|_{Z_1^*}\\
		&\qquad=
		\sup_{u \in \calA}
		\frac{\left|
		\sum_{\alpha \in A} \EE
		\Big\langle
		\EE\big[ \|W f\|_{\ell^2(A \times B)}^{-1} \ind{U} W_\alpha^2 f_{\alpha} \big| \calF_\alpha \big],
		u_\alpha \EE\big[\|Wf\|_{\ell^2(A \times B)} \big|\calF_\alpha\big]\Big\rangle_{\ell^2(B)}
       	\right|}{\EE \Big(\sum_{\alpha \in A} 
		\| \pi_\alpha u_\alpha \|_{\ell^2(B)}^2 
		\EE\big[\|W f\|_{\ell^2(A \times B)} \big|\calF_\alpha \big]^2 \Big)^\frac{1}{2}
		}.
	\end{align*}
	By Cauchy--Schwarz inequality and \eqref{eq:35}, we get
	\begin{align*}
		&\EE \Big(\sum_{\alpha \in A}
		\| \pi_\alpha u_\alpha \|_{\ell^2(B)}^2 \EE\big[\|W f\|_{\ell^2(A \times B)} 
		\big|\calF_\alpha \big]^2 \Big)^\frac{1}{2} \\
		&\qquad\qquad\leq
		\EE\Big(\|u\|_{\ell^2(A \times B)} \cdot 
		\sup_{\alpha \in A} \EE\big[\|W f\|_{\ell^2(A \times B)} \big| \calF_\alpha \big]\Big) \\
		&\qquad\qquad\leq
		\delta^{-1} \|u\|_{Y_2} \|Wf\|_{L^2(\ell^2(A\times B))}.
	\end{align*}
	Therefore,
	\begin{align*}
		&
		\left\|
		\left(
		\EE\big[\|W f\|_{\ell^2(A \times B)}^{-1} \ind{U}
		W_{\alpha}^2 f_{\alpha} \big| \calF_\alpha \big] : \alpha \in A\right)
		\right\|_{Y_1^*} \\
		&\geq
		\delta
		\sup_{u \in \calA}
		\frac{\left|
		\sum_{\alpha \in A} \EE
		\left\langle
		\EE\big[ \|W f\|_{\ell^2(A \times B)}^{-1} \ind{U} W_\alpha^2 f_{\alpha} \big| \calF_\alpha \big] \cdot
		\EE\big[ \|Wf\|_{\ell^2(A \times B)} \big|\calF_\alpha\big],
		u_\alpha 
		\right\rangle_{\ell^2(B)}
       	\right|}{\|u\|_{Y_2} \|Wf\|_{L^2(\ell^2(A \times B))}} \\
		&=
		\frac{\delta}{\|Wf\|_{L^2(\ell^2(A\times B))}}
		\left\| 
		\left(
		\EE\big[ \|W f\|_{\ell^2(A \times B)}^{-1} \ind{U} W_\alpha^2 f_{\alpha} \big| \calF_\alpha \big] 
		\cdot
		\EE\big[ \|Wf\|_{\ell^2(A \times B)} \big|\calF_\alpha\big]
		:
		\alpha \in A
		\right)
		\right\|_{Y_2}.
	\end{align*}
	Since $f \in \calA$, by Cauchy--Schwarz inequality, we have
	\begin{align*}
		\kappa^2 |f_{\alpha, \beta}| 
		&\leq
		\EE
		|f_{\alpha, \beta}| \cdot \EE\big[ w_{\alpha, \beta} \big| \calF_\alpha \big]^2 \\
		&=
		|f_{\alpha, \beta}| \cdot \EE\big[ w_{\alpha, \beta} \ind{U} \big| \calF_\alpha \big]^2 \\
		&\leq
		|f_{\alpha, \beta}| \cdot 
		\EE\big[\|W f \|_{\ell^2(A \times B)} \big| \calF_\alpha \big]
		\cdot
		\EE\Big[ \|W f\|_{\ell^2(A \times B)}^{-1} \ind{U} w_{\alpha, \beta}^2 \big| \calF_\alpha \big],
	\end{align*}
	thus
	\[
		\left\|
		\left(
		\EE\big[ \|W f\|_{\ell^2(A \times B)}^{-1} \ind{U} W_\alpha^2 f_{\alpha} \big| \calF_\alpha \big]
		\cdot
		\EE\big[ \|Wf\|_{\ell^2(A \times B)} \big|\calF_\alpha\big]
		:
		\alpha \in A
		\right)
		\right\|_{Y_2}
		\geq
		\kappa^2
		\|f\|_{Y_2}.
	\]
	Hence,
	\[
		\left\|
		\left(
		\EE\big[\|W f\|_{\ell^2(A \times B)}^{-1} \ind{U}
		W_{\alpha}^2 f_{\alpha} \big| \calF_\alpha \big] : \alpha \in A\right)
		\right\|_{Y_1^*}
		\geq \delta 
		\|Wf\|_{L^2(\ell^2(A\times B))}^{-1} \kappa^2 \|f\|_{Y_2} \geq \kappa^2 \delta,
	\]
	which completes the proof of the lemma.
	\end{proof}
	Having justified Lemma \ref{lem:2}, we complete the proof of the theorem.
\end{proof}

\section*{Acknowledgments}
The research was partially supported by the National Science Centre, Poland, Grant number 2016/23/B/ST1/01665.

\begin{bibliography}{f4}
	\bibliographystyle{amsplain}

\providecommand{\bysame}{\leavevmode\hbox to3em{\hrulefill}\thinspace}
\providecommand{\MR}{\relax\ifhmode\unskip\space\fi MR }
\providecommand{\MRhref}[2]{%
  \href{http://www.ams.org/mathscinet-getitem?mr=#1}{#2}
}
\providecommand{\href}[2]{#2}
\begin{thebibliography}{10}

\bibitem{Aum}
R.J. Aumann, \emph{Borel structures for function spaces}, Illinois J. Math.
  \textbf{5} (1961), no.~4, 614--630.

\bibitem{b1}
J.~Brossard, \emph{Comparaison des {”normes”} {L}p du processus croissant
  et de la variable maximale pour les martingales r\'eguli\`eres \`a deux
  indices. {T}h\'eor\`eme local correspondant}, Ann. Probab. \textbf{8} (1980),
  1183--1188.

\bibitem{CW}
R.~Cairoli and J.B. Walsh, \emph{Stochastic integrals in the plane}, Acta Math.
  \textbf{134} (1975), 111--183.

\bibitem{Ch}
M.~Christ, \emph{A {$T(b)$} theorem with remarks on analytic capacity and the
  {C}auchy integral}, Colloq. Math. \textbf{60/61} (1990), 601--628.

\bibitem{GdlP}
V.~de~la Peña and E.~Giné, \emph{Decoupling: From dependence to
  independence}, Springer Science \& Business Media, 2012.

\bibitem{Feff}
R.~Fefferman, \emph{Singular integrals on product domains}, Bull. Amer. Math.
  Soc. \textbf{4} (1981), 195--201.

\bibitem{FeffPiph}
R.~Fefferman and J.~Pipher, \emph{Multiparameter operators and sharp weighted
  inequalities}, Amer. J. Math. \textbf{119} (1997), 337--369.

\bibitem{FeffSt}
R.~Fefferman and E.M. Stein, \emph{Singular integrals on product spaces}, Adv.
  in Math. \textbf{45} (1982), 117--143.

\bibitem{Gar}
A.M. Garsia, \emph{Martingale inequalities: {S}eminar notes on recent
  progress}, Math. Lecture Note, W.A. Benjamin Inc., 1973.

\bibitem{Gundy}
R.F. Gundy, \emph{In\'egaliti\'es pour les martingales \`a un et deux indices:
  L’espace {$H^p$}}, Lect. Notes Math., vol. 774, Springer-Verlag, 1978.

\bibitem{GunSt}
R.F. Gundy and E.M. Stein, \emph{{$H^p$} spaces for the bidisc}, P. Natl. Acad.
  Sci. USA \textbf{76} (1978), 106--109.

\bibitem{HyKa}
T.~Hyt\"onen and A.~Kairema, \emph{What is a cube{?}}, Ann. Acad. Sci. Fenn.-M.
  \textbf{38} (2013), 405--412.

\bibitem{Imk}
P.~Imkeller, \emph{Two-parameter martingales and their quadratic variation},
  Lecture Notes in Mathematics, Springer, 1988.

\bibitem{JSch}
W.B. Johnson and G.~Schechtman, \emph{Sums of independent random variables in
  rearrangement invariant function spaces}, Ann. Probab. \textbf{17} (1989),
  789--808.

\bibitem{KW}
S.~Kwapień and W.~Woyczynski, \emph{Random series and stochastic integrals:
  Single and multiple}, Birkhauser Boston, 1992.

\bibitem{MS}
S.~Montgomery-Smith, \emph{Concrete representation of martingales}, Electron.
  J. Probab. \textbf{3} (1998), 1--15.

\bibitem{NaRiSt}
A.~Nagel, F.~Ricci, and E.M. Stein, \emph{Singular integrals with flag kernels
  and analysis on quadratic {CR} manifolds}, J. Funct. Anal. \textbf{181}
  (2001), 29--118.

\bibitem{NaSt}
A.~Nagel and E.M. Stein, \emph{On the product theory of singular integrals},
  Rev. Mat. Iberoamericana \textbf{20} (2004), 531--561.

\bibitem{NaWa}
A.~Nagel and S.~Wainger, \emph{{$L^2$} boundedness of {H}ilbert transforms
  along surfaces and convolution operators homogeneous with respect to a
  multiple parameter group}, Amer. J. Math. \textbf{99} (1977), 761--785.

\bibitem{Peter}
S.~Petermichl, \emph{Dyadic shifts and a logarithmic estimate for {H}ankel
  operators with matrix symbol}, C. R. Acad. Sci. Paris S\'er. I Math.
  \textbf{330} (2000), 455--460.

\bibitem{RiSt}
F.~Ricci and E.M. Stein, \emph{Multiparameter singular integrals and maximal
  functions}, Ann. Inst. Fourier (Grenoble) \textbf{42} (1992), 637--670.

\bibitem{mrz}
M.~Rzeszut, \emph{Comparison between {$L^1$} norms of square function and
  maximal function of an {$(F_4)$} martingale}, preprint, 2019.

\bibitem{StTr}
T.~Steger and B.~Trojan, \emph{Littlewood--{P}aley theory for triangle
  buildings}, J. Geom. Anal. \textbf{28} (2018), 1122--1150.

\bibitem{StSt}
E.M. Stein and B.~Street, \emph{Multi-parameter singular {R}adon transforms
  {II}: the {Lp} theory}, Adv. Math. \textbf{248} (2013), 736--783.

\bibitem{St}
B.~Street, \emph{Multi-parameter singular {R}adon transforms {I}: the {$L^2$}
  theory}, J. Anal. Math. \textbf{116} (2012), 83--162.

\bibitem{weisz1}
F.~Weisz, \emph{Martingale {H}ardy spaces and their applications in {F}ourier
  analysis}, Lect. Notes Math., vol. 1568, Springer--Verlag, 1996.

\bibitem{weisz2}
\bysame, \emph{Summability of multi-dimensional {F}ourier series and {H}ardy
  spaces}, Math. and Its Appl., vol. 541, Springer Netherlands, 2002.

\bibitem{WoZa}
E.~Wong and M.~Zakai, \emph{Martingales and stochastic integrals for processes
  with a multidimnensional parameter}, Z. Wahrscheinlichkeitstheorie und Verw.
  Gebiete \textbf{29} (1974), 109--122.

\end{thebibliography}
\end{bibliography}

\end{document}